\documentclass[a4paper,12pt]{article}

\usepackage{dsfont, amsmath, amsthm, amsfonts, amssymb,accents}

\theoremstyle{plain}
\newtheorem{theorem}{Theorem}[section]

\newtheorem{lemma}{Lemma}[section]

\theoremstyle{remark}
\newtheorem{remark}{Remark}[section]

\numberwithin{equation}{section}

\def\tht{\theta}
\def\Om{\Omega}

\def\e{\varepsilon}
\def\g{\gamma}
\def\G{\Gamma}
\def\l{\lambda}
\def\p{\partial}
\def\D{\Delta}

\def\E{\mbox{\rm e}}
\def\a{\alpha}

\def\d{\delta}
\def\L{\Lambda}

\def\vp{\varphi}

\def\Odr{\mathcal{O}}
\def\H{W_2}

\def\Ho{\mathring{W}_2}
\def\Hoper{\mathring{W}_{2,per}}

\def\di{\,\mathrm{d}}

\def\iu{\mathrm{i}}

 \DeclareMathOperator{\RE}{Re}
\DeclareMathOperator{\IM}{Im} \DeclareMathOperator{\spec}{\sigma}

\DeclareMathOperator{\essspec}{\sigma_{e}}

\DeclareMathOperator{\Dom}{\mathcal{D}}

\begin{document}
\allowdisplaybreaks

\title{Homogenization of the planar waveguide with frequently
alternating boundary conditions}

\author{D. Borisov$^1$, and G. Cardone$^2$}
\date{}
\maketitle

\begin{quote}
{\small {\em 1) Bashkir State Pedagogical University, October
Revolution
\\
\phantom{1) } St.~3a, 450000 Ufa, Russia
\\
2) University of Sannio, Department of Engineering,
\\
\phantom{2) } Corso Garibaldi, 107, 84100 Benevento, Italy }
\\
\phantom{1) }\texttt{borisovdi@yandex.ru},
\texttt{giuseppe.cardone@unisannio.it}}
\end{quote}

\begin{quote}
{\small We consider Laplacian in a planar strip with Dirichlet
boundary condition on the upper boundary and with frequent
alternation boundary condition on the lower boundary. The
alternation is introduced by the periodic partition of the boundary
into small segments on which Dirichlet and Neumann conditions are
imposed in turns. We show that under the certain condition the
homogenized operator is the Dirichlet Laplacian and prove the
uniform resolvent convergence. The spectrum of the perturbed
operator consists of its essential part only and has a band
structure. We construct the leading terms of the asymptotic
expansions for the first band functions. We also construct the
complete asymptotic expansion for the bottom of the spectrum. }
\end{quote}

\begin{quote}
\end{quote}

\section{Introduction}

The model of quantum waveguides with window(s) was studied in a
series of papers by several authors, see \cite{JPA-B}, \cite{BGRS},
\cite{DK}, \cite{ESTV}, \cite{EV2}, \cite{G}, \cite{HTW},
\cite{B-MS06}, \cite{BEG-02}. Such waveguides were modeled by a pair
of two planar strips or three-dimensional layers having common
boundary and the window(s) are  openings of finite size in it. The
usual operator is the Dirichlet Laplacian. The main interest is the
behavior of the spectrum of such operator and its dependence on the
window. If the strips or layers are of the same width, then the
problem reduces to the Laplacian in one strip, and the window is
modeled by segment(s) on the boundary where the Dirichlet condition
switches to the Neumann one. This  model poses interesting
mathematical questions, and it is also of physical interests, since
it has certain applications in nanophysical devices and in modeling
electromagnetic waveguides.

It was shown in the above cited papers that the perturbation by a
finite number of the windows leaves the essential spectrum unchanged
and give rise to new discrete eigenvalues emerging below the
threshold of the essential spectrum. This phenomenon was studied,
and the behavior of the emerging eigenvalues was described.

A completely different situation occurs, if one deals with an
infinite number of the windows located on the boundary. In this case
the perturbation is not localized and as a result it changes an
essential spectrum. Exactly this situation is considered in this
paper. Namely, we consider a planar strip with periodically located
windows of the same length. The windows are modeled by segments
where the Dirichlet boundary condition is replaced by the Neumann
one. The main feature is that the sizes of these windows are small
and the distance between each two neighboring windows is small, too.
Such perturbation is well-known in homogenization theory, see, for
instance, \cite{MZ-B}, \cite{Fr}, \cite{Ch}, \cite{Izv03-B},
\cite{VMU-B}, \cite{BLP}, \cite{Dlcr}. In the case of the bounded
domains it is known that under certain condition on the alternation
the homogenized operator is the Dirichlet Laplacian, i.e., the
homogenized boundary condition is the Dirichlet one. The same
phenomenon occurs in our problem. 
In other words, in the limit
the perturbed operator behaves as if there are no windows at all.
Moreover, it happens even in the case when the size of the windows
are relatively larger than the remaining parts with Dirichlet
condition, see condition (\ref{1.3}) and Theorem~\ref{th1.1}.

The above mentioned convergence of the perturbed operator is in the
uniform resolvent sense. It also holds true, if we consider the
resolvent not only as an operator in $L_2$ but as those from $L_2$
into $\H^1$. We give an effective estimate for the rate of the
convergence. Such kind of estimates for the operators with fast
oscillating coefficients were obtained recently in the series of
papers \cite{BS2}, \cite{BS5}, \cite{BoAA08}, \cite{Zh3},
\cite{Zh4}, \cite{Zh5}, \cite{PT}, \cite{Pas}. Although the
perturbation by fast oscillating coefficients is also typical for
the homogenization theory and it has a number of features similar to
the perturbation by frequent alternation of boundary condition, in
our case the situation is rather different from that in the cited
papers. Namely, while considering the resolvent as an operator from
$L_2$ into $\H^1$, they had to introduce a special corrector to get
an estimate for the rate of convergence. In our case we do not need
such a corrector, and the estimate for the rate of the convergence
can be obtained in a rather easy way exactly for the difference of
the resolvents. This is a specific feature of the problems of
boundary homogenization and it was known before in the case of the
homogenization of the fast oscillating boundary in the case of a
bounded domain, see \cite[Ch. I\!I\!I, Sec. 4.1]{OIS}.

One more result of our paper concerns the behavior of spectrum of
the perturbed operator. The spectrum has the band structure and we
describe the asymptotic behavior for the first band functions w.r.t.
a small parameter. It implies that the length of the first band
tends to infinity w.r.t. a small parameter and therefore all
possible gaps ``run'' to infinity. We prove that the bottom of the
spectrum corresponds to a periodic eigenfunction of the operator
obtained by Floquet decomposition of the periodic operator. On the
base of this fact we obtain the complete asymptotic expansion of the
bottom of the spectrum.

In conclusion we describe briefly the contents of the article. In
the next section we formulate the problem and give the main results.
In the third section we prove the uniform resolvent convergence of
the perturbed operator. The fourth section is devoted to a similar
result but for the operator on a periodicity cell obtained in the
Floquet decomposition. In the last, fifth section we analyze the
bottom of the spectrum of the perturbed operator.

\section{Formulation of the problem and the main results}

Let $x=(x_1,x_2)$ be Cartesian coordinates in $\mathds{R}^2$, $\e$
be a small positive parameter, $\eta=\eta(\e)$ be a function
satisfying the estimate
\begin{equation}\label{1.1}
0<\eta(\e)<\frac{\pi}{2}
\end{equation}
for all $\e$. We partition the real axis into two subsets,
\begin{equation*}
 \g_\e:=\{x: |x_1-\e\pi m|<\e\eta,\ m\in \mathds{Z}, x_2=0\},\quad
\G_\e:=Ox_1\setminus\overline{\g}_\e.
\end{equation*}
By $\Om$, $\G_+$, and $\G_-$ we denote the strip $\{x: 0<x_2<\pi\}$
and its upper and lower boundary, respectively.

The main object of our study is the Laplacian in $L_2(\Om)$ subject
to the Dirichlet boundary condition on $\G_+\cup\g_\e$ and to the
Neumann one on $\G_\e$. Rigorously we introduce it as the
self-adjoint operator in $L_2(\Om)$ associated with a sesquilinear
form
\begin{equation}\label{1.2}
\mathfrak{h}_\e[u,v]:=(\nabla u,\nabla v)_{L_2(\Om)}\quad
\text{on}\quad \Ho^1(\Om,\G_+\cup\g_\e),
\end{equation}
where $\Ho^1(Q,S)$ indicates the subset of the functions in
$\H^1(Q)$ having zero trace on the curve $S$. We will employ the
symbol $\mathcal{H}_\e$ to denote this operator.

\begin{remark}\label{rm1.1}
Although it is not one of the main issues of our paper, it is
possible to describe explicitly the structure of the functions in
the domain of $\mathcal{H}_\e$. More precisely, it is possible to
describe their behavior at the end-points of $\g_\e$. We refer to
Lemma~\ref{lm2.1} for more details.
\end{remark}

The main aim of the paper is to study the behavior of the resolvent
and of the spectrum of $\mathcal{H}_\e$ as $\e\to+0$. We introduce
one more self-adjoint operator $\mathcal{H}_0$ which is the
Dirichlet Laplacian in $L_2(\Om)$. We define it as associated with a
sesquilinear form
\begin{equation*}
\mathfrak{h}_0[u,v]:=(\nabla u,\nabla v)_{L_2(\Om)}\quad
\text{on}\quad \Ho^1(\Om,\p\Om).
\end{equation*}
It is well-known that the domain of this operator is
$\H^2(\Om)\cap\Ho^1(\Om,\p\Om)$. In what follows the symbol
$\|\cdot\|_{A\to B}$ indicates the norm of an operator from the
space $A$ to $B$.

Our first result says that under the condition
\begin{equation}\label{1.3}
 \e\ln\eta(\e)\to 0,\quad \e\to+0,
\end{equation}
the operator $\mathcal{H}_0$ is the homogenized one for
$\mathcal{H}_\e$.

\begin{theorem}\label{th1.1}
Suppose (\ref{1.3}). Then the estimate
\begin{equation}\label{1.4}
\|(\mathcal{H}_\e-\iu)^{-1}-(\mathcal{H}_0-\iu)^{-1}\|_{L_2(\Om)\to\H^1(\Om)}
\leqslant 9\, \e^{1/4}|\ln\sin\eta(\e)|^{1/4}
\end{equation}
holds true. 
\end{theorem}

As it follows from (\ref{1.3}), the quantity $\e|\ln\sin\eta(\e)|$
tends to zero as $\e\to+0$. Even if $\eta$ tends to zero not very
fast, say, as $\eta\sim\e^\a$, $\a>0$, and the lengths of the
Dirichlet parts on $\G_-$  are therefore relatively small with
respect to those of the Neumann parts, the homogenized operator is
still subject to Dirichlet condition on $\G_-$. This fact was known
in the case of bounded domains, see, for instance, \cite{Fr},
\cite{Ch}. Moreover, if $\eta\to\frac{\pi}{2}-0$ as $\e\to+0$, then
the measures of Neumann parts of the boundary are relatively small
w.r.t. to those of Dirichlet parts. In this case
$|\ln\sin\eta|\to+0$ and it improves the rate of the convergence in
(\ref{1.4}).

The spectrum of $\mathcal{H}_0$ consists only of its essential
component and coincides with the semi-axis $[1,+\infty)$. As a
corollary of Theorem~\ref{th1.1} we have

\begin{theorem}\label{th1.2}
The spectrum of $\mathcal{H}_\e$ converges to that of
$\mathcal{H}_0$. Namely, if $\l\not\in [1,+\infty)$, then
$\l\not\in\spec(\mathcal{H}_\e)$ for $\e$ small enough. And if
$\l\in[1,+\infty)$, then there exists
$\l_\e\in\spec(\mathcal{H}_\e)$  so that $\l_\e\to\l$ as $\e\to+0$.
\end{theorem}

The operator $\mathcal{H}_\e$ is a periodic one due to the
periodicity of the sets $\g_\e$ and $\G_\e$, and its spectrum has a
band structure. Namely, let
\begin{align*}
&\Om_\e:=\{x: |x_1|<\e\pi, \ 0<x_2<\pi\},\quad
\mathring{\g}_\e:=\p\Om_\e\cap\g_\e,
\\
&\mathring{\G}_\e:=\p\Om_\e\cap\G_\e,\quad
\mathring{\G}_\pm:=\p\Om_\e\cap\G_\pm.
\end{align*}
By $\mathcal{H}_\e^{(p)}(\tau)$ we denote the self-adjoint operator
in $L_2(\Om_\e)$ associated with the sesquilinear form
\begin{equation*}
\mathring{\mathfrak{h}}_\e^{(p)}[u,v]:=\left( \left(\iu\frac{\p}{\p
x_1}-\frac{\tau}{\e}\right)u,\left(\iu\frac{\p}{\p
x_1}-\frac{\tau}{\e}\right)v\right)_{L_2(\Om_\e)}+ \left(\frac{\p
u}{\p x_2},\frac{\p u}{\p x_2}\right)_{L_2(\Om_\e)}
\end{equation*}
on $\Hoper^1(\Om_\e,\mathring{\G}_+\cup\mathring{\g}_\e)$, where
$\tau\in[-1,1)$. Here
$\Hoper^1(\Om_\e,\mathring{\G}_+\cup\mathring{\g}_\e)$ is the subset
of the functions in
$\Ho^1(\Om_\e,\mathring{\G}_+\cup\mathring{\g}_\e)$ satisfying
periodic boundary conditions on the lateral boundaries of $\Om_\e$.
Since the domain $\Om_\e$ is bounded, the operator
$\mathcal{H}_\e^{(p)}(\tau)$ has a compact resolvent and its
spectrum consists of a countably many discrete eigenvalues
accumulating at infinity. We denote these eigenvalues by
$\l_n(\tau,\e)$ and arrange them in the non-descending order with
the multiplicity taking into account
\begin{equation*}
\l_1(\tau,\e)\leqslant\l_2(\tau,\e)\leqslant \l_3(\tau,\e)\leqslant
\ldots \leqslant \l_n(\tau,\e)\leqslant \ldots
\end{equation*}
Let $\spec(\cdot)$, $\essspec(\cdot)$ be the spectrum and the
essential spectrum of an operator. Then
\begin{equation}\label{1.4a}
\spec(\mathcal{H}_\e)=\essspec(\mathcal{H}_\e)=\bigcup\limits_{n=1}^\infty
\{\l_n(\tau,\e): \tau\in[-1,1)\}
\end{equation}
that 
will be shown in Lemma~\ref{lm3.0}.

The rest of the results is devoted to the behavior of
$\l_n(\tau,\e)$ as $\e\to+0$. First we establish a uniform resolvent
convergence for $\mathcal{H}_\e^{(p)}(\tau)$.

By $\mathfrak{L}$ we denote the subspace of the functions in
$L_2(\Om_\e)$ which are independent of $x_1$, and we decompose
$L_2(\Om_\e)$ as follows
\begin{equation}\label{1.6}
L_2(\Om_\e)=\mathfrak{L}\oplus \mathfrak{L}^\bot,
\end{equation}
where $\mathfrak{L}^\bot$ indicates the orthogonal complement to
$\mathfrak{L}$ in $L_2(\Om_\e)$. In $\mathfrak{L}$ we introduce a
self-adjoint operator $\mathcal{Q}$ as associated with a
sesquilinear form
\begin{equation*}
\mathfrak{q}[u,v]:=\left(\frac{d u}{dx_2},\frac{d
v}{dx_2}\right)_{L_2(0,\pi)} \quad \text{on}\quad
\Ho^1((0,\pi),\{0,\pi\}).
\end{equation*}
In other words, $\mathcal{Q}$ is the operator $-\frac{d^2}{dx_2^2}$
in $L_2(0,\pi)$ subject to the Dirichlet boundary condition. Denote
$\mathcal{H}_0^{(p)}:=\mathcal{Q}\oplus 0$, where $0$ indicates the
zero operator on $\mathfrak{L}^\bot$.

\begin{theorem}\label{th1.3}
Let $|\tau|<1-\d$, where $0<\d<1$ is a fixed constant and assume
(\ref{1.3}). Then the convergence
\begin{equation}\label{1.7}
\left\|
\left(\mathcal{H}_\e^{(p)}(\tau)-\frac{\tau^2}{\e^2}\right)^{-1} -
\big(\mathcal{H}_0^{(p)}\big)^{-1} \right\|_{L_2(\Om_\e)\to
L_2(\Om_\e)}\leqslant
\frac{\e+5\e^{1/2}|\ln\sin\eta|^{1/2}}{\d^{1/2}}
\end{equation}
holds true, where the constant $C$ is independent of $\e$ and
$\eta$.
\end{theorem}

The resolvent
$\left(\mathcal{H}_\e^{(p)}(\tau)-\frac{\tau^2}{\e^2}\right)^{-1}$
is well-defined that will be shown in the proof of
Lemma~\ref{lm3.1}.

We should mention that the results of Theorem~\ref{th1.3} are close
to those of Theorem~1.2 in \cite{FS}. Moreover, the technique we
employ to prove Theorem~\ref{th1.3} is similar to that proposed in
\cite{FS}. The next theorem should be regarded as the corollary of
Theorem~\ref{th1.3}.

\begin{theorem}\label{th1.4}
Let the hypothesis of Theorem~\ref{th1.3} holds. Then given any $N$
there exists $\e_0>0$ such that for $\e<\e_0$, $n\leqslant N$ the
eigenvalues $\l_n(\tau,\e)$ satisfy the asymptotics
\begin{equation}\label{1.8}
\begin{aligned}
&\l_n(\tau,\e)=\frac{\tau^2}{\e^2}+n^2+R_n(\tau,\e),\quad \e\to+0,
\\
&|R_n(\tau,\e)|\leqslant n^4
\frac{\sqrt{2}\e+8\e^{1/2}|\ln\sin\eta|^{1/2}}{\d^{1/4}}.
\end{aligned}
\end{equation}
\end{theorem}

The last theorem implies that the length of the first $N$ bands of
the spectrum $\{\l_n(\tau,\e): \tau\in[-1,1)\}$, $n=1,\ldots,N$ are
of order at least $\Odr(\e^{-2})$. Moreover, they overlap. It means
that the first zone of the spectrum stretches as $\e\to+0$ and in
the limit it coincides with the semi-axis $[1,+\infty)$. It implies
that all possible gaps in the spectrum of $\mathcal{H}_\e$ ``run''
to infinity with the speed at least $\Odr(\e^{-2})$. This is a
natural situation for the homogenization problems, see for instance,
\cite{JMS-B}, \cite{Bir1}.

The bottom of the spectrum of $\mathcal{H}_\e$ is given by
$\inf\limits_{\tau\in[-1,1)} \l_1(\tau,\e)$ and by Theorem~\ref{1.8}
it converges to one as $\e\to+0$. The next theorem gives its
complete asymptotic expansion as $\e\to+0$.

\begin{theorem}\label{th1.5}
The first eigenvalue $\l_1(\tau,\e)$ attains its infimum at
$\tau=0$. The asymptotics
\begin{align}
&\l_1(0,\e)=1+\sum\limits_{j=1}^{\infty}\e^j\mu_j(\eta),\label{1.9}
\\
&\mu_1(\eta)=\frac{2}{\pi}\ln\sin\eta(\e),\quad
\mu_2(\eta)=\frac{3}{\pi^2}\ln^2\sin\eta(\e),\label{1.10}
\end{align}
holds true, and other $\mu_j$ are determined in a recurrent way by
 (\ref{4.15}). Moreover,
\begin{equation}\label{1.11}
\mu_j(\eta)=K_j\ln^j\eta+\Odr(\ln^{j-3}\eta),\quad \eta\to+0,
\end{equation}
where $K_j$ are some constants.
\end{theorem}

We observe that due to (\ref{1.11}) the coefficients $\mu_j$ has
increasing logarithmic singularities as $\eta\to+0$. At the same
time, the terms of the series (\ref{1.9}) behave as
$\Odr(\e^i\ln^i\eta)$, if $\eta\to+0$ as $\e\to+0$, and in view of
the condition (\ref{1.3}) the series (\ref{1.9}) remains an
asymptotic one. We note that this phenomenon for the problems in the
bounded domains with the frequent alternation of the boundary
conditions was described first in \cite{MZ-B}, \cite{VMU-B}.

Theorem~\ref{th1.4} does not describe all the eigenvalues of the
operator $\mathcal{H}_\e$. Namely, we conjecture that there exists
two-parametric family of the eigenvalues of $\mathcal{H}_\e$
behaving as
\begin{equation*}
\l_{n,m}(\tau,\e)\sim \frac{(\tau+2m)^2}{\e^2}+n^2+\ldots,\qquad
m\in\mathds{Z},\quad n\in\mathds{N}.
\end{equation*}
The reason for such conjecture is that the right-hand side of this
relation is in fact the eigenvalues of the operator
\begin{equation*}
\left(\iu\frac{\p}{\p x_1}-\frac{\tau}{\e}\right)^2-\frac{\p^2}{\p
x_2^2}
\end{equation*}
in $L_2(\Om_\e)$ subject to the Dirichlet boundary condition on
$\mathring{\G}_+\cup\mathring{\G}_-$ and to the periodic boundary
condition on the lateral boundaries of $\Om_\e$. Such operator
appears, if one treats $\mathcal{H}_0$ as periodic w.r.t. $x_1$ and
makes the Floquet decomposition. Moreover, it is natural to expect
that the same formulas are valid not for $|\tau|<1-\d$, as in
Theorem~\ref{th1.4}, but for all $\tau\in[-1,1)$. Such formulas
would allow to answer one more interesting question on the presence
or absence of the gaps in the spectrum of $\mathcal{H}_\e$. As we
said above, if exist, such gaps ``run'' to infinity as $\e\to+0$.
Generally speaking, the length of the lengths of gaps could be
small, finite or infinite as $\e\to+0$. At the same time, in
\cite{JMS-B} it was shown that the spectrum of a periodic
one-dimensional Schr\"odinger
\begin{equation*}
-\frac{d^2}{dx^2}+a\left(\frac{x}{\e}\right) \quad \text{in}\quad
L_2(\mathds{R})
\end{equation*}
can contains only the gaps of finite or small lengths. So, it allows
us to conjecture the same for $\mathcal{H}_\e$ provided the gaps
exist.

\section{Convergence of the resolvent of $\mathcal{H}_\e$}

This section is devoted to the proof of
Theorems~\ref{th1.1},~\ref{th1.2}.

Let $\chi=\chi(t)\in C^\infty(\mathds{R})$ be a cut-off function
with values in $[0,1]$ equalling  one as $t<1$ and vanishing as
$t>2$. By $\Dom(\cdot)$ we denote the domain of an operator.

We indicate by $(r^{(m)}_\pm,\tht^{(m)}_\pm)$ the polar coordinates
centered at $(\e\pi m\pm\e\eta,0)$, $m\in\mathds{Z}$, so that
$\tht^{(m)}_\pm=0$ corresponds to the points of $\g_\e$.

\begin{lemma}\label{lm2.1}
Each function $u\in\Dom(\mathcal{H}_\e)$ can be represented as
\begin{align}
&u(x)=\accentset{0}{u}(x)+\accentset{1}{u}(x),\label{2.0}
\\
&\accentset{0}{u}(x)=\sum\limits_{m\in\mathds{Z}} \a_\pm^{(m)}
\sqrt{r^{(m)}_\pm}\chi\left(\frac{3r^{(m)}_\pm}{\e\d_\e}\right)
\sin\frac{\tht^{(m)}_\pm}{2},\nonumber
\\
&\d_\e:=\min\left\{\eta(\e),\frac{\pi}{2}-\eta(\e)\right\},\nonumber
\end{align}
where $\a^{(m)}_\pm$ are some constants, and
$\accentset{1}{u}\in\H^2(\Om)\cap \Ho^1(\Om,\G_+\cup\g_\e)$. The
estimate
\begin{equation}\label{2.0a}
\sum\limits_{m\in\mathds{Z}} \left( |\a^{(m)}_+|^2+ |\a^{(m)}_-|^2
\right)+\|\accentset{1}{u}\|_{\H^2(\Om)}^2\leqslant C
\|\mathcal{H}_\e u\|_{L_2(\Om)}^2
\end{equation}
holds true, where the constant $C$ is independent of $u$.
\end{lemma}

\begin{proof}
The domain of $\mathcal{H}_\e$ consists of the generalized solutions
$u\in\H^1(\Om)$ to the problem
\begin{equation}\label{2.20}
-\D u=f,\quad \text{in}\quad \Om,\qquad
u=0\quad\text{on}\quad\G_+\cup\g_\e,\qquad \frac{\p u}{\p
x_2}=0\quad\text{on}\quad x\in\G_\e.
\end{equation}
It follows that
\begin{equation}\label{2.21}
\mathfrak{h}_\e[u,u]=(f,u).
\end{equation}
Since $u=0$ on $\G_+$, the first eigenvalue of $-\frac{d^2}{dx_2^2}$
on the cross-section of $\Om$ is at least $1/4$. This is why
\begin{equation}\label{2.22}
\Big\|\frac{\p u}{\p x_2}(x_1,\cdot)\Big\|_{L_2(0,\pi)}^2\geqslant
\frac{1}{4}\|u(x_1,\cdot)\|_{L_2(0,\pi)}^2.
\end{equation}
Hence,
\begin{equation*}
\mathfrak{h}_\e[u,u]\geqslant \Big\|\frac{\p u}{\p
x_2}\Big\|_{L_2(\Om_\e)}^2\geqslant
\frac{1}{4}\|u\|_{L_2(\Om_\e)}^2,\quad\mathcal{H}_\e\geqslant 1/4,
\end{equation*}
and it follows from (\ref{2.21}) that
\begin{equation}\label{2.23}
\|u\|_{L_2(\Om)}\leqslant 4\|f\|_{L_2(\Om)},\quad \|\nabla
u\|_{L_2(\Om)}\leqslant 2\|f\|_{L_2(\Om)}.
\end{equation}
Employing these inequalities and proceeding as in the proof of
Theorem~2.1 in \cite{JPA-B}, one can prove easily the representation
(\ref{2.0}), and the estimates
\begin{gather*}
\|\accentset{1}{u}\|_{\H^2(\Om)}\leqslant C\|\mathcal{H}_\e
u\|_{L_2(\Om)},
\\
|\a_\pm^{(m)}|\leqslant C \left( \|\mathcal{H}_\e
u\|_{L_2(\{x\in\Om:\ r_\pm^{(m)}<\e\d_\e\})}+ \|
u\|_{\H^1(\{x\in\Om:\ r_\pm^{(m)}<\e\d_\e\})} \right),
\end{gather*}
where the constant $C$ is independent of $u$ and $m$. Summing up the
last inequalities, we arrive at (\ref{2.0a}).
\end{proof}

We introduce an auxiliary function
\begin{align}
&X=X(\xi,\eta)=\RE\ln \big(\sin z+\sqrt{\sin^2 z-\sin^2\eta}
\big)-\xi_2,\label{2.7}
\\
& \xi=(\xi_1,\xi_2)=\left(\frac{x_1}{\e},\frac{x_2}{\e}\right),\quad
z=\xi_1+\iu\xi_2,\nonumber
\end{align}
where the branches of the logarithm and the root are specified by
the requirements $\ln 1=0$, $\sqrt{1}=1$. This function was
introduced in \cite{AA-GRR} and it was shown that it is harmonic as
$\xi_2>0$, even and $\pi$-periodic w.r.t. $\xi_1$, decays
exponentially as $\xi_2\to+\infty$, and satisfies the boundary
conditions
\begin{equation}\label{2.8}
X=\ln\sin\eta\quad\text{on}\quad \g(\eta),\qquad \frac{\p
X}{\p\xi_2}=-1\quad \text{on}\quad \G(\eta),
\end{equation}
where
\begin{equation}\label{2.8a}
\g(\eta):=\{\xi: |\xi_1-\pi m|<\eta,\ m\in \mathds{Z},\
\xi_2=0\},\quad \G(\eta):=O\xi_1\setminus\overline{\g(\eta)}.
\end{equation}
The function $X$ is continuous in $\{\xi: \xi_2\geqslant 0\}$ and
satisfies the estimate
\begin{equation}\label{2.9}
|X|\leqslant |\ln\sin\eta|
\end{equation}
uniformly in $\xi$. Indeed, since it is harmonic and decays
exponentially as $\xi_2\to+\infty$, it achieves its maximum on
$O\xi_1$. Employing this fact and the explicit formula for $X$, one
can easily check the estimate (\ref{2.9}).

\begin{lemma}\label{lm2.2}
Given any $u\in\Dom(\mathcal{H}_\e)$, the function $u
X\left(\frac{\cdot}{\e},\eta\right)$ belongs to
$\Ho^1(\Om,\G_+\cup\g_\e)$.
\end{lemma}
\begin{proof}
The boundary conditions and the belongings $u X, X\nabla u\in
L_2(\Om)$ are due to the belonging $u\in\Ho^1(\Om,\G_+\cup\g_\e)$
and the estimate (\ref{2.9}). It remains to check that $u\nabla X\in
L_2(\Om)$. We employ the representation (\ref{2.0}) for $u$ and due
to (\ref{2.0a}) we obtain $\accentset{0}{u}\nabla X\in L_2(\Om)$. To
prove the belonging $\accentset{1}{u} \nabla X\in L_2(\Om)$, we
integrate by parts taking into account the properties of $X$,
\begin{align*}
&\int\limits_{\Om} \nabla X\cdot |\accentset{1}{u}|^2\nabla X\di x=
-\int\limits_{\G_-} X |\accentset{1}{u}|^2\frac{\p X}{\p x_2}\di
x_1- \int\limits_{\Om} X\nabla X\cdot \nabla |\accentset{1}{u}|^2
\di x
\\
&= \frac{1}{\e} \int\limits_{\G_\e} X|\accentset{1}{u}|^2 \di x
-\frac{1}{2} \int\limits_{\Om} \nabla X^2\cdot\nabla
|\accentset{1}{u}|^2 \di x
\\
&=\frac{1}{\e} \int\limits_{\G_\e} X|\accentset{1}{u}|^2 \di x +
\frac{1}{2} \int\limits_{\G_-} X^2\frac{\p |\accentset{1}{u}|^2}{\p
x_2}\di x_1+ \frac{1}{2}\int\limits_{\Om} X^2 \D|\accentset{1}{u}|^2
\di x
\\
&=\frac{1}{\e} \int\limits_{\G_\e} X|\accentset{1}{u}|^2 \di x +
\RE\int\limits_{\G_\e} X^2 \accentset{1}{u}\frac{\p
\accentset{1}{u}}{\p x_2}\di x_1+ \frac{1}{2}\int\limits_{\Om} X^2
\D|\accentset{1}{u}|^2 \di x,
\end{align*}
which by the estimate (\ref{2.9}) and the belongings
\begin{equation*}
\accentset{1}{u}\in\H^2(\Om),\quad \accentset{1}{u}, \frac{\p
\accentset{1}{u}}{\p x_2}\in L_2(\G_\e)
\end{equation*}
implies $\accentset{1}{u}\nabla X\in L_2(\Om)$.
\end{proof}

\begin{proof}[Proof of Theorem~\ref{th1.1}]
Denote $u_\e:=(\mathcal{H}_\e-\iu)^{-1}f$ for $f\in L_2(\Om)$. By
the definition of $\mathcal{H}_\e$, the function $u_\e$ satisfies
the identity
\begin{equation}\label{2.10}
(\nabla u_\e,\nabla \phi)_{L_2(\Om)}+\iu
(u_\e,\phi)_{L_2(\Om)}=(f,\phi)_{L_2(\Om)}
\end{equation}
for any $\phi\in\Ho^1(\Om,\G_+\cup\g_\e)$. For $\phi=u_\e$ we have
\begin{equation}\label{2.1}
\|\nabla u_\e\|_{L_2(\Om)}^2+\iu \|u_\e\|_{L_2(\Om)}^2 =
(f,u_\e)_{L_2(\Om)}.
\end{equation}
We take the imaginary part of the last identity and obtain
\begin{align}
&\|u_\e\|_{L_2(\Om)}^2=\IM (f,u_\e)_{L_2(\Om)}\leqslant
\|f\|_{L_2(\Om)} \|u_\e\|_{L_2(\Om)},\nonumber
\\
&\|u_\e\|_{L_2(\Om)}\leqslant \|f\|_{L_2(\Om)}. \label{2.3}
\end{align}
It follows from (\ref{2.1}), (\ref{2.3}) that
\begin{equation}\label{2.4}
\|\nabla u_\e\|_{L_2(\Om)}^2=\RE (f,u_\e)_{L_2(\Om)}\leqslant
\|f\|_{L_2(\Om)}^2.
\end{equation}
In the same way for $u_0:=(\mathcal{H}_0-\iu)^{-1}f$ we have the
inequalities
\begin{equation}\label{2.4a}
\|u_0\|_{L_2(\Om)}\leqslant \|f\|_{L_2(\Om)},\quad \|\nabla
u_0\|_{L_2(\Om)}\leqslant \|f\|_{L_2(\Om)}.
\end{equation}
By Lemma~\ref{lm2.2} the function $\phi=u_\e X$ belongs to
$\Ho^1(\Om,\G_+\cup\g_\e)$. We substitute it into (\ref{2.10}),
\begin{equation}\label{2.11}
 (\nabla u_\e,X\nabla u_\e)_{L_2(\Om)} + (\nabla u_\e,
u_\e \nabla X)_{L_2(\Om)} +\iu (u_\e,X u_\e)_{L_2(\Om)}= (f,
Xu_\e)_{L_2(\Om)}.
\end{equation}
We integrate by parts and employ the properties of $X$ and
(\ref{2.8}),
\begin{align*}
&\RE (\nabla u_\e, u_\e \nabla X)_{L_2(\Om)} = \frac{1}{2}
\int\limits_{\Om} \nabla X \cdot (u_\e\nabla \overline{u}_\e+
\overline{u}_\e\nabla u_\e) \di x
\\
&= \frac{1}{2}\int\limits_{\Om} \nabla X\cdot \nabla |u_\e|^2\di x =
- \frac{1}{2}\int\limits_{\G_-} |u_\e|^2\frac{\p X}{\p x_2} \di x_1
- \int\limits_{\Om} |u_\e|^2\D X\di x
\\
&=\frac{1}{2\e} \int\limits_{\G_\e} |u_\e|^2\di x_1.
\end{align*}
Now taking the real part of (\ref{2.11}), we arrive at the identity
\begin{equation}\label{2.12}
(\nabla u_\e,X\nabla u_\e)_{L_2(\Om)} + \frac{1}{2\e}
\|u_\e\|_{L_2(\G_\e)}^2= \RE (f,X u_\e)_{L_2(\Om)}.
\end{equation}
By (\ref{2.9}), (\ref{2.3}), (\ref{2.4}) it yields
\begin{align}
&\frac{1}{2\e} \|u_\e\|_{L_2(\G_\e)}^2 \leqslant \RE (f,X
u_\e)_{L_2(\Om)} \leqslant |\ln\sin\eta|
\|f\|_{L_2(\Om)}^2,\nonumber
\\
&\|u_\e\|_{L_2(\G_-)}=\|u_\e\|_{L_2(\G_\e)}\leqslant
\sqrt{2\e|\ln\sin\eta(\e)|}\|f\|_{L_2(\Om)}.\label{2.13}
\end{align}

Denote $v_\e:=u_\e-u_0$. This function belongs to
$\Ho^1(\Om,\G_+\cup\g_\e)$ and is a generalized solution to the
problem
\begin{gather*}
-\D v_\e+\iu v_\e=0\quad \text{in}\quad \Om,
\\
v_\e=0\quad \text{on}\quad \G_+\cup\g_\e,\qquad \frac{\p v_\e}{\p
x_2}=-\frac{\p u_0}{\p x_2} \quad \text{on}\quad \G_\e.
\end{gather*}
We multiply the equation by $\overline{v}_\e$ and integrate by
parts,
\begin{align}
&-\int\limits_{\G_\e} \overline{v}_\e \frac{\p u_0}{\p x_2}\di x_1+
\|\nabla v_\e\|_{L_2(\Om)}^2 +\iu \|v_\e\|_{L_2(\Om)}^2=0,
\\
&
\begin{aligned}
&\|\nabla v_\e\|_{L_2(\Om)}^2=\RE\int\limits_{\G_\e} \overline{v}_\e
\frac{\p u_0}{\p x_2}\di x_1=\RE\int\limits_{\G_\e} \overline{u}_\e
\frac{\p u_0}{\p x_2}\di x_1
\\
&\hphantom{\|\nabla v_\e\|_{L_2(\Om)}^2} \leqslant
\|u_\e\|_{L_2(\G_\e)} \Big\|\frac{\p u_0}{\p x_2}\Big\|_{L_2(\G_-)},
\\
&\|v_\e\|_{L_2(\Om)}^2\leqslant \IM \int\limits_{\G_\e}
\overline{u}_\e \frac{\p u_0}{\p x_2}\di x_1\leqslant
\|u_\e\|_{L_2(\G_\e)} \Big\|\frac{\p u_0}{\p x_2}\Big\|_{L_2(\G_-)}.
\end{aligned}\label{2.16}
\end{align}
Let us estimate $\big\|\frac{\p u_0}{\p x_2}\big\|_{L_2(\G_-)}$. For
a.e. $x_1\in \mathds{R}$ we have
\begin{equation*}
\frac{\p u_0}{\p x_2}(x_1,0)= \frac{1}{\pi} \int\limits_{0}^{\pi}
\frac{\p}{\p x_2}(x_2-\pi) \frac{\p u}{\p x_2}\di x_2.
\end{equation*}
By Cauchy-Schwarz inequality we derive
\begin{equation}\label{2.4b}
\begin{aligned}
\left|\frac{\p u_0}{\p x_2}(x_1,0)\right|^2&\leqslant
\frac{2}{\pi^2} \Bigg( \int\limits_{0}^{\pi} (x_2-\pi)^2\di x_2
\int\limits_{0}^{\pi} \left|\frac{\p^2 u_0}{\p x_2^2}(x)\right|^2\di
x_2
\\
&\hphantom{\leqslant \frac{2}{\pi^2} \Bigg(}
+\int\limits_{0}^{\pi}\di x_2 \int\limits_{0}^{\pi} \left|\frac{\p
u_0}{\p x_2}(x)\right|^2\di x_2\Bigg)
\\
&= \frac{2}{\pi} \left( \frac{\pi^2}{3} \Big\|\frac{\p^2 u_0}{\p
x_2^2}(x_1,0)\Big\|_{L_2(0,\pi)}^2 + \Big\|\frac{\p u_0}{\p
x_2}(x_1,0)\Big\|_{L_2(0,\pi)}^2\right).
\end{aligned}
\end{equation}
Proceeding as in the proof of Lemma~7.1 in \cite[Ch.~3, Sec.~7]{Ld},
we check that
\begin{equation*}
\Big\|\frac{\p^2 u_0}{\p x_1^2}\Big\|_{L_2(\Om)}^2+ \Big\|\frac{\p^2
u_0}{\p x_2^2}\Big\|_{L_2(\Om)}^2+ 2\Big\|\frac{\p^2 u_0}{\p x_1 \p
x_2}\Big\|_{L_2(\Om)}^2= \|f-\iu u_0\|_{L_2(\Om)}^2,
\end{equation*}
and by (\ref{2.4a}) it implies
\begin{equation*}
\Big\|\frac{\p^2 u_0}{\p x_2^2}\Big\|_{L_2(\Om)}\leqslant \|f-\iu
u_0\|_{L_2(\Om)} \leqslant 2\|f\|_{L_2(\Om)}.
\end{equation*}
This estimate and (\ref{2.4b}) yield
\begin{equation*}
\Big\|\frac{\p u_0}{\p x_2}\Big\|_{L_2(\G_-)}^2 \leqslant
\frac{2\pi}{3} \Big\|\frac{\p^2 u_0}{\p x_2^2}\Big\|_{L_2(\Om)}^2+
\frac{2}{\pi} \| u_0\|_{L_2(\Om)}^2\leqslant
\frac{8\pi^2+6}{3\pi}\|f\|_{L_2(\Om)}^2.
\end{equation*}
Substituting this estimate and (\ref{2.13}) into (\ref{2.16}), we
get
\begin{equation*}
\|v_\e\|_{\H^1(\Om)}^2\leqslant 4 \sqrt{\frac{4\pi^2+3}{3\pi}}
\sqrt{\e|\ln\sin\eta(\e)|} \|f\|_{L_2(\Om)}^2 \leqslant
9\sqrt{\e|\ln\sin\eta(\e)|}\|f\|_{L_2(\Om)}^2
\end{equation*}
that completes the proof.
\end{proof}

Theorem~\ref{th1.2} follows directly from Theorem~\ref{th1.1} and
Theorems V\!I\!I\!I.23, V\!I\!I\!I.24 in \cite[Ch. V\!I\!I\!I, Sec.
7]{RS}.

\section{Convergence of the resolvent of $\mathcal{H}_\e^{(p)}(\tau)$}

In this section we prove Theorems~\ref{th1.3},~\ref{th1.4}. We begin
with auxiliary lemmas.

\begin{lemma}\label{lm3.0}
The identity (\ref{1.4a}) holds true.
\end{lemma}

\begin{proof}
Given $\l_n(\tau,\e)$, let $\psi_n(x,\tau,\e)$ be the associated
eigenfunction. Employing the function $\E^{\frac{\iu\tau}{x_1}\e}
\psi_n(x,\tau,\e)$, one can construct easily a singular sequence for
$\mathcal{H}_\e$ at $\l=\l_n(\tau,\e)$ and by Weyl criterion we
therefore obtain
\begin{equation*}
\bigcup\limits_{n=1}^\infty \{\l_n(\tau,\e):
\tau\in[-1,1)\}\subseteq\essspec(\mathcal{H}_\e).
\end{equation*}

Let
\begin{equation}\label{3.32}
\l\not\in\bigcup\limits_{n=1}^\infty \{\l_n(\tau,\e):
\tau\in[-1,1)\}.
\end{equation}
It sufficient to prove that $\l\not\in\spec(\mathcal{H}_\e)$. It is
equivalent to the existence of the resolvent
$(\mathcal{H}_\e-\l)^{-1}$. Let us prove the latter.

We introduce the Gelfand transformation $\mathcal{F}_\e$ w.r.t.
$x_1$,
\begin{align*}
&\big(\mathcal{F}_\e f\big)(x,\tau)=\E^{-\frac{\iu\tau
x_1}{\e}}\big(\mathcal{G}_\e f\big)(x,\tau),
\hphantom{x_2)\E^{\frac{\iu\tau m}{\e}}}\quad x\in\Om_\e,
\\
&\big(\mathcal{G}_\e f\big)(x,\tau):=\sum\limits_{m\in\e\pi
\mathds{Z}} f(x_1-m,x_2)\E^{\frac{\iu\tau m}{\e}}, \quad x\in\Om_\e,
\\
&\big(\mathcal{F}_\e^{-1}\widehat{f}\big)(x):=\frac{1}{2}\int\limits_0^2
\widehat{f}(x,\tau)\E^{\frac{\iu\tau x_1}{\e}} \di\tau,
\hphantom{x_2)\E^{\frac{\tau}{\e}}} \quad x\in\Om,
\end{align*}
where it is assumed in the last formula that the functions defined
on $\Om_\e$ are extended $\e\pi$-periodically w.r.t. $x_1$.

Let $X$ be a Hilbert space, and define
\begin{equation*}
L_2((0,2),X):=\int\limits_{(0,2)}^{\oplus} X.
\end{equation*}
Repeating the proof of Theorem~2.2.5 in \cite[Ch. 2, Sec. 2.2]{Ku},
one can prove easily that $\mathcal{G}_\e: L_2(\Om)\to
L_2((0,2),L_2(\Om_\e))$ is an isomorphism, and
\begin{equation}\label{3.30}
\|\mathcal{G}_\e
f\|_{L_2((0,2),L_2(\Om_\e))}^2=2\|f\|_{L_2(\Om)},\quad
(\mathcal{G}_\e f, \mathcal{G}_\e
g)_{L_2((0,2),L_2(\Om_\e))}=2(f,g)_{L_2(\Om)}.
\end{equation}
Similarly, $\mathcal{G}_\e: \Ho^1(\Om,\G_+\cup\g_\e)\to
L_2((0,2),\Hoper^1(\Om_\e,\mathring{\G}_+\cup\mathring{\g}_\e))$ is
an isomorphism, and
\begin{equation}\label{3.31}
\begin{aligned}
&\|\mathcal{G}_\e f\|_{L_2((0,2),\Hoper^1(\Om_\e,
\mathring{\G}_+\cup\mathring{\g}_\e))}^2
=2\|f\|_{\Ho^1(\Om,\G_+\cup\g_\e)},
\\
&(\mathcal{G}_\e f, \mathcal{G}_\e
g)_{L_2((0,2),\Hoper^1(\Om_\e,\mathring{\G}_+\cup\mathring{\g}_\e))}
=2(f,g)_{\Ho^1(\Om,\G_+\cup\g_\e)}.
\end{aligned}
\end{equation}

Given $f\in L_2(\Om)$, let $\widehat{f}_\e(x,\tau):=(\mathcal{F}_\e
f)(x,\tau)$. Due to (\ref{3.32}), the operator $\big(
\mathcal{H}_\e^{(p)}(\tau)-\l\big)^{-1}$ is invertible for each
$\tau\in[-1,1)$. The function
$\widehat{u}_\e=\widehat{u}_\e(x,\tau)$,
\begin{equation*}
\widehat{u}_\e:=\big( \mathcal{H}_\e^{(p)}(\tau)-\l\big)^{-1} \in
 \widehat{f} \Hoper^1(\Om_\e,\mathring{\G}_+\cup\mathring{\g}_\e)
\end{equation*}
satisfies the uniform in $\tau$ estimate
\begin{equation*}
\|\widehat{u}_\e\|_{\H^1(\Om_\e)}\leqslant
C\|\widehat{f}\|_{L_2(\Om_\e)}.
\end{equation*}
Hence, it belongs to
$L_2((0,2),\Hoper^1(\Om_\e,\mathring{\G}_+\cup\mathring{\g}_\e))$,
and by (\ref{3.31}) the function
\begin{equation*}
u_\e(x):=\big( \mathcal{F}_\e^{-1}
\widehat{u}_\e\big)(x)=\frac{1}{2}\int\limits_{0}^{2}
\widehat{u}_\e(x,\tau)\E^{\frac{\iu \tau x_1}{\e}}\di\tau
\end{equation*}
belongs to $\Ho^1(\Om,\G_+\cup\g_\e)$.

Given any $\vp\in\Ho^1(\Om,\G_+\cup\g_\e)$, denote
$\widehat{\vp}:=\mathcal{F}_\e\vp$. The identities (\ref{3.30}),
(\ref{3.31}) and the definition of $\widehat{u}_\e$ yield
\begin{align*}
\mathfrak{h}_\e[u,\vp]-\l(u,\vp)_{L_2(\Om)}&= \frac{1}{2} \left(
\left(\iu\frac{\p}{\p x_1}-\frac{\tau}{\e}\right) \widehat{u}_\e,
\left(\iu\frac{\p}{\p x_1}-\frac{\tau}{\e}\right) \widehat{\vp}_\e
\right)_{L_2((0,2),L_2(\Om_\e))}
\\
&+\frac{1}{2} \left( \frac{\p\widehat{u}_\e}{\p x_2},\frac{\p
\widehat{\vp}_\e}{\p
x_2}\right)_{L_2((0,2),L_2(\Om_\e))}-\frac{\l}{2} (\widehat{u}_\e,
\widehat{\vp}_\e)_{L_2((0,2),L_2(\Om_\e))}
\\
&=\frac{1}{2}
(\widehat{f}_\e,\widehat{\vp}_\e)_{L_2((0,2),L_2(\Om_\e))}=(f,\vp)_{L_2(\Om)}.
\end{align*}
Thus, $u=(\mathcal{H}_\e-\l)^{-1}f$ and the operator
$\mathcal{H}_\e-\l$ is boundedly invertible.
\end{proof}

\begin{lemma}\label{lm3.1}
Let $|\tau|<1-\d$, where $0<\d<1$ is a fixed constant, and
\begin{equation*}
u_\e=\left(\mathcal{H}_\e^{(p)}(\tau)-\frac{\tau^2}{\e^2}\right)^{-1}f.
\end{equation*}
Then
\begin{align}
&\|u_\e\|_{L_2(\Om_\e)}\leqslant 4\|f\|_{L_2(\Om_\e)},\label{3.0a}
\\
&\Big\|\frac{\p u_\e}{\p x_2}\Big\|_{L_2(\Om_\e)}\leqslant
2\|f\|_{L_2(\Om_\e)},\label{3.0b}
\\
&\Big\|\frac{\p u_\e}{\p x_1}\Big\|_{L_2(\Om_\e)}\leqslant
\frac{2}{\d^{1/2}}\|f\|_{L_2(\Om_\e)}.\label{3.0c}
\end{align}
If, in addition, $f\in \mathfrak{L}^\bot$, then
\begin{equation}\label{3.1}
\|u_\e\|_{L_2(\Om_\e)}\leqslant
\frac{\e}{\d^{1/2}}\|f\|_{L_2(\Om_\e)},\quad \|\nabla
u_\e\|_{L_2(\Om_\e)}\leqslant \frac{\e}{2\d}  \|f\|_{L_2(\Om_\e)}.
\end{equation}
\end{lemma}

\begin{proof}
Let us prove first that the resolvent
$\left(\mathcal{H}_\e^{(p)}(\tau)-\frac{\tau^2}{\e^2}\right)^{-1}$
is well-defined. The quadratic form corresponding to
$\mathcal{H}_\e^{(p)}(\tau)-\frac{\tau^2}{\e^2}$ reads as follows,
\begin{equation}\label{3.2}
\begin{aligned}
\left( \left(\mathcal{H}_\e^{(p)}(\tau)-\frac{\tau^2}{\e^2}\right)u,
u\right)_{L_2(\Om_\e)}=&\Big\|\left(\iu\frac{\p}{\p
x_1}-\frac{\tau}{\e}\right)u\Big\|_{L_2(\Om_\e)}^2
\\
&-\frac{\tau^2}{\e^2} \|u\|_{L_2(\Om_\e)}^2 +\Big\|\frac{\p u}{\p
x_2}\Big\|_{L_2(\Om_\e)}^2
\end{aligned}
\end{equation}
on $\Hoper^1(\Om_\e,\mathring{\G}_+\cup\mathring{\g}_\e)$. We can
expand $u(\cdot,x_2)$ in terms of the basis $\{\E^{\pm \frac{2\iu m
x_1}{\e}}\}$, $m=0,1,2,\ldots$ Employing this expansion, one can
make sure that
\begin{align}
&
\begin{aligned}
&\Big\|\left(\iu\frac{\p}{\p
x_1}-\frac{\tau}{\e}\right)u\Big\|_{L_2(\Om_\e)}^2-\frac{\tau^2}{\e^2}
\|u\|_{L_2(\Om_\e)}^2
\\
&=\Big\|\left(\iu\frac{\p}{\p
x_1}-\frac{\tau}{\e}\right)u^\bot\Big\|_{L_2(\Om_\e)}^2-\frac{\tau^2}{\e^2}
\|u^\bot\|_{L_2(\Om_\e)}^2
\\
&\geqslant \frac{4(1-|\tau|)}{\e^2} \|u^\bot\|_{L_2(\Om_\e)}^2
\geqslant \frac{4\d}{\e^2}\|u^\bot\|_{L_2(\Om_\e)}^2,
\end{aligned}
\label{3.4}
\\
&\Big\|\left(\iu\frac{\p}{\p
x_1}-\frac{\tau}{\e}\right)u\Big\|_{L_2(\Om_\e)}^2-\frac{\tau^2}{\e^2}
\|u\|_{L_2(\Om_\e)}^2\geqslant\d\Big\|\frac{\p u}{\p
x_1}\Big\|_{L_2(\Om_\e)}^2,\label{3.5}
\end{align}
where $u^\bot$ is the projection of $u$ on $\mathfrak{L}^\bot$. It
follows from (\ref{2.22}) that
\begin{equation}\label{3.6}
\Big\|\frac{\p u}{\p x_2}\Big\|_{L_2(\Om_\e)}^2\geqslant
\frac{1}{4}\|u\|_{L_2(\Om_\e)}^2
\end{equation}
for $u\in\Ho^1(\Om_\e,\mathring{\G}_+)$. The estimates (\ref{3.6})
and (\ref{3.4}) imply that
\begin{equation*}
\mathcal{H}_\e^{(p)}(\tau)-\frac{\tau^2}{\e^2}\geqslant \frac{1}{4},
\end{equation*}
and therefore the inverse of this operator is well-defined and
satisfies the estimate (\ref{3.0a}). In view of (\ref{3.2}) we thus
have
\begin{equation}\label{3.8}
\Big\|\left(\iu\frac{\p}{\p
x_1}-\frac{\tau}{\e}\right)u_\e\Big\|_{L_2(\Om_\e)}^2-\frac{\tau^2}{\e^2}
\|u_\e\|_{L_2(\Om_\e)}^2+ \Big\|\frac{\p u_\e}{\p
x_2}\Big\|_{L_2(\Om_\e)}^2 = (f,u_\e)_{L_2(\Om_\e)}.
\end{equation}
This identity, (\ref{3.0a}), and (\ref{3.5}) imply
\begin{equation*}
\Big\|\frac{\p u_\e}{\p x_2} \Big\|_{L_2(\Om_\e)}^2\leqslant
4\|f\|_{L_2(\Om_\e)}^2,\quad \d \Big\|\frac{\p u_\e}{\p x_1}
\Big\|_{L_2(\Om_\e)}^2\leqslant 4\|f\|_{L_2(\Om_\e)}
\end{equation*}
that proves (\ref{3.0b}), (\ref{3.0c}).

Assume that $f\in \mathfrak{L}^\bot$ and let $u_\e^\bot$ be the
projection of $u_\e$ on $\mathfrak{L}^\bot$. Then it follows from
(\ref{3.8}), (\ref{3.4}) that
\begin{equation*}
\Big\|\left(\iu\frac{\p}{\p
x_1}-\frac{\tau}{\e}\right)u_\e^\bot\Big\|_{L_2(\Om_\e)}^2-\frac{\tau^2}{\e^2}
\|u_\e^\bot\|_{L_2(\Om_\e)}^2+ \Big\|\frac{\p u_\e}{\p
x_2}\Big\|_{L_2(\Om_\e)}^2 = (f,u_\e^\bot)_{L_2(\Om_\e)}.
\end{equation*}
We substitute the estimate (\ref{3.4}) into the last identity,
\begin{align}
&\frac{4\d}{\e^2}\|u_\e^\bot\|_{L_2(\Om_\e)}^2\leqslant
\|f\|_{L_2(\Om_\e)}\|u_\e^\bot\|_{L_2(\Om_\e)},\nonumber
\\
& \|u_\e^\bot\|_{L_2(\Om_\e)}\leqslant \frac{\e^2}{4\d}
\|f\|_{L_2(\Om_\e)}, \label{3.9}
\\
&|(f,u_\e)_{L_2(\Om_\e)}|=|(f,u_\e^\bot)_{L_2(\Om_\e)}|\leqslant
\frac{\e^2}{4\d} \|f\|_{L_2(\Om_\e)}^2.
\end{align}
The last estimate, (\ref{3.5}), (\ref{3.6}), (\ref{3.8}) yield
\begin{equation*}
\|u_\e\|_{L_2(\Om_\e)}^2\leqslant \frac{\e^2}{\d}
\|f\|_{L_2(\Om_\e)}^2,\quad \|\nabla u_\e\|_{L_2(\Om_\e)}^2\leqslant
\frac{\e^2}{4\d^2} \|f\|_{L_2(\Om_\e)}^2
\end{equation*}
that completes the proof.
\end{proof}

\begin{lemma}\label{lm3.2}
Let $F=F(x_2)\in L_2(0,\pi)$ and $U:=\mathcal{Q}^{-1}F$. Then
\begin{equation}\label{3.10}
|U'(0)|\leqslant \sqrt{\frac{\pi}{3}} \|F\|_{L_2(0,\pi)}.
\end{equation}
\end{lemma}

\begin{proof}
It is easy to find the function $U$ explicitly,
\begin{equation*}
U(x_2)=-\frac{1}{2}\int\limits_{0}^{\pi} \left(|x_2-t|-x_2-t+\frac{2
x_2 t}{\pi}\right)F(t)\di t.
\end{equation*}
Hence,
\begin{equation*}
U'(0)=\int\limits_{0}^{\pi} \left(1-\frac{t}{\pi}\right) F(t)\di t,
\end{equation*}
and by Cauchy-Schwarz inequality we obtain
\begin{equation*}
|U'(0)|\leqslant \left(\int\limits_{0}^{\pi}
\left(1-\frac{t}{\pi}\right)^2\di
t\right)^{1/2}\|F\|_{L_2(0,\pi)}=\sqrt{\frac{\pi}{3}}\|F\|_{L_2(0,\pi)}.
\end{equation*}
\end{proof}



\begin{lemma}\label{lm3.4}
Each function $u\in\Dom(\mathcal{H}_\e^{(p)}(\tau))$ can be
represented as
\begin{align}
&u(x)=\accentset{0}{u}(x)+\accentset{1}{u}(x),\label{3.14}
\\
&\accentset{0}{u}(x)=\a_-\chi\left(\frac{3r_-^{(0)}}{\e\d_\e}\right)
\sqrt{r_-^{(0)}}\sin\frac{\tht_-^{(0)}}{2} +
\a_+\chi\left(\frac{3r_+^{(0)}}{\e\d_\e}\right)
\sqrt{r_+^{(0)}}\sin\frac{\tht_+^{(0)}}{2},\nonumber
\end{align}
where $\accentset{1}{u}(x)\in\H^2(\Om_\e)$ vanishes on
$\mathring{G}_+\cup\mathring{\g}_\e$ and satisfies periodic boundary
condition on the lateral boundaries of $\Om_\e$. Here $\a_\pm$ are
some constants and $\d_\e$ is the same is in Lemma~\ref{lm2.1}.
\end{lemma}

The proof is completely analogous to that of Lemma~\ref{lm2.1}.

\begin{proof}[Proof of Theorem~\ref{th1.3}]
Given $f\in L_2(\Om_\e)$, we decompose it as $f=F_\e+f_\e^\bot$,
$F_\e\in\mathfrak{L}$, $f_\e^\bot\in \mathfrak{L}^\bot$,
\begin{equation}\label{3.12}
\|F_\e\|_{L_2(\Om_\e)}^2+\|f_\e^\bot\|_{L_2(\Om_\e)}^2=\|f\|_{L_2(\Om_\e)}^2.
\end{equation}
Then
\begin{equation*}
\left(\mathcal{H}_\e^{(p)}(\tau)-\frac{\tau^2}{\e^2}\right)^{-1}f=
\left(\mathcal{H}_\e^{(p)}(\tau)-\frac{\tau^2}{\e^2}\right)^{-1}F_\e+
\left(\mathcal{H}_\e^{(p)}(\tau)-\frac{\tau^2}{\e^2}\right)^{-1}f_\e^\bot,
\end{equation*}
and by (\ref{3.1}), (\ref{3.2}), (\ref{3.12}) we obtain immediately
\begin{equation}
\left\|\left(\mathcal{H}_\e^{(p)}(\tau)-\frac{\tau^2}{\e^2}\right)^{-1}f_\e^\bot
\right\|_{L_2(\Om_\e)} \leqslant \frac{\e}{\d^{1/2}}
\|f_\e^\bot\|_{L_2(\Om_\e)} \leqslant \frac{\e}{\d^{1/2}}
\|f\|_{L_2(\Om_\e)},\label{3.15a}
\end{equation}
It remains to construct an appropriate approximation for
\begin{equation*}
u_\e:=\left(\mathcal{H}^{(p)}_\e(\tau)-\frac{\tau^2}{\e^2}\right)^{-1}F_\e.
\end{equation*}
It is clear that
$\big(\mathcal{H}_0^{(p)}\big)^{-1}f=\mathcal{Q}^{-1}F_\e$. We
denote this function by $U_\e$. Let $\chi$ be a cut-off function
defined before Lemma~\ref{lm2.1}. We introduce one more function,
\begin{equation*}
\widehat{u}_\e(x):=U_\e(x_2)+\e U_\e'(0)
\left(X\left(\frac{x}{\e},\eta(\e)\right)-\ln\sin\eta(\e)\right)\chi(x_2).
\end{equation*}
It is straightforward to check that $\widehat{u}_\e$ satisfies
periodic boundary condition on the lateral surfaces of $\Om_\e$,
vanishes on $\mathring{\G}_+\cup\mathring{\g}_\e$, and obeys Neumann
condition on $\mathring{\G}_\e$.
It also belongs to the domain of the operator
$\mathcal{H}_\e^{(p)}(\tau)$ since the function
$\chi(x_2)X\left(\frac{x}{\e},\eta(\e)\right)$ satisfies the
representation (\ref{3.14}).

Employing the properties of $X$, we see that
\begin{equation*}
\left(\mathcal{H}^{(p)}_\e(\tau)-\frac{\tau^2}{\e^2}\right)\widehat{u}_\e=
F_\e-U_\e'(0) \left(2\iu\tau\frac{\p X}{\p x_1}+ \e\chi''
(X-\ln\sin\eta)+2\e\chi'\frac{\p X}{\p x_2}\right),
\end{equation*}
and for $\widetilde{u}_\e:=u_\e-\widehat{u}_\e$ we have
\begin{align*}
\widetilde{u}_\e=&U'_\e(0)
\left(\mathcal{H}^{(p)}_\e(\tau)-\frac{\tau^2}{\e^2}\right)^{-1}g_\e-
\e U'_\e(0)\ln\sin\eta
\left(\mathcal{H}^{(p)}_\e(\tau)-\frac{\tau^2}{\e^2}\right)^{-1}\chi''
\\
= &\widetilde{u}_\e^{(1)} + \widetilde{u}_\e^{(2)},
\\
g_\e=&2\iu\tau\frac{\p X}{\p x_1}\chi+\e \left(\chi''
X+2\chi'\frac{\p X}{\p x_2}\right).
\end{align*}
It follows from \cite[Lm. 3.7]{VMU-B} that
\begin{equation*}
\int\limits_{-\pi/2}^{\pi/2} X(\xi,\eta)\di
\xi_1=0\quad\text{for}\quad \xi_2>0.
\end{equation*}
Hence,
\begin{equation*}
\int\limits_{-\e\pi/2}^{\e\pi/2}
X\left(\frac{x}{\e},\eta(\e)\right)\di x_1=0\quad\text{for}\quad
0<x_2<\pi,
\end{equation*}
and $g\in \mathfrak{L}^\bot$. By (\ref{3.1}) it implies that
\begin{equation}\label{3.19}
\|\widetilde{u}_\e^{(1)}\|_{L_2(\Om_\e)} \leqslant
\frac{\e}{\d^{1/2}}|U'_\e(0)| \|g\|_{L_2(\Om_\e)}.
\end{equation}
The identity
\begin{equation*}
\|\nabla_\xi X\|_{L_2(\Pi)}^2=\pi|\ln\sin\eta|,\quad \Pi:=\{\xi:
|\xi_1|<\pi/2, \xi_2>0\},
\end{equation*}
was proven in \cite[Lm. 3.8]{VMU-B}. Together with (\ref{2.9}),
(\ref{1.3}) it yields
\begin{equation}\label{3.20}
\begin{aligned}
\|g\|_{L_2(\Om_\e)}&\leqslant 2\Big\|\frac{\p X}{\p
x_1}\Big\|_{L_2(\Om_\e)}
+\e|\ln\sin\eta|\|\chi''\|_{L_2(\Om_\e)}+2\e C\Big\|\frac{\p X}{\p
x_2}\Big\|_{L_2(\Om_\e)}
\\
&\leqslant 2 \sqrt{2} \|\nabla_x X\|_{L_2(\Om_\e)}
+C\e^{3/2}|\ln\sin\eta|
\\
& \leqslant 2\sqrt{2}\|\nabla_\xi
X\|_{L_2(\Pi)}+C\e^{3/2}|\ln\sin\eta|
\\
& =2\sqrt{2\pi} |\ln\sin\eta|^{1/2}+C\e^{3/2}|\ln\sin\eta|\leqslant
6|\ln\sin\eta|^{1/2},
\end{aligned}
\end{equation}
if $\e$ is small enough. Here $C$ is a constant independent of $\e$
and $\eta$. Since
\begin{equation*}
F_\e(x_2)=(\e\pi)^{-1}
\int\limits_{-\frac{\e\pi}{2}}^{\frac{\e\pi}{2}} f(x)\di x_1,
\end{equation*}
by Cauchy-Schwarz inequality we have
\begin{equation*}
\|F_\e\|_{L_2(0,\pi)}\leqslant (\e\pi)^{-1/2} \|f\|_{L_2(\Om_\e)}.
\end{equation*}
This estimate and Lemma~\ref{lm3.2} yield
\begin{equation}\label{3.16}
|U_\e'(0)|\leqslant \frac{\e^{-1/2}}{\sqrt{3}} \|f\|_{L_2(\Om_\e)}.
\end{equation}
It follows from (\ref{3.19}), (\ref{3.20}), and the last estimate
that
\begin{equation}\label{3.21}
\|\widetilde{u}_\e^{(1)}\|_{L_2(\Om_\e)}\leqslant
\frac{2\sqrt{3}}{\d^{1/2}}\e^{1/2}|\ln\sin\eta|^{1/2}\|f\|_{L_2(\Om_\e)}.
\end{equation}
Since $\|\chi''\|_{L_2(\Om_\e)}=\e^{1/2}\|\chi''\|_{L_2(0,\pi)}$, by
(\ref{3.0a}) and (\ref{3.16}) we derive
\begin{equation*}
\|\widetilde{u}_\e^{(2)}\|_{L_2(\Om_\e)}\leqslant C\e|\ln\sin\eta|
\|f\|_{L_2(\Om_\e)},
\end{equation*}
where $C$ is a constant independent of $\e$ and $\eta$. Thus,
\begin{equation}\label{3.22}
\|\widetilde{u}_\e\|_{L_2(\Om_\e)} \leqslant
\frac{4\e^{1/2}}{\d^{1/2}}|\ln\sin\eta|^{1/2} \|f\|_{L_2(\Om_\e)},
\end{equation}
if $\e$ is small enough. It follows from (\ref{3.16}) and
(\ref{2.9}) that
\begin{equation*}
\|\e U'_\e(0)(X-\ln\sin\eta)\chi\|_{L_2(\Om_\e)}\leqslant
\frac{2\e\pi}{\sqrt{3}}|\ln\sin\eta|\|f\|_{L_2(\Om_\e)}.
\end{equation*}
Hence, by (\ref{3.15a}), (\ref{3.22}), (\ref{1.3}) we obtain
\begin{align*}
&
\bigg\|\left(\mathcal{H}_\e^{(p)}(\tau)-\frac{\tau^2}{\e^2}\right)^{-1}f
- \left(\mathcal{H}_0^{(p)}\right)^{-1}f\bigg\|_{L_2(\Om_\e)}
\\
&\hphantom{\mathcal{H}_\e^{(p)}(\tau)-}\leqslant\left(
\frac{\e}{\d^{1/2}}+\frac{4}{\d^{1/2}}\e^{1/2}|\ln\sin\eta|^{1/2}
+\frac{2\pi}{\sqrt{3}}\e|\ln\sin\eta|\right)\|f\|_{L_2(\Om_\e)}
\\
&\hphantom{\mathcal{H}_\e^{(p)}(\tau)-}\leqslant
\frac{\e+5\e^{1/2}|\ln\sin\eta|^{1/2}}{\d^{1/2}}\|f\|_{L_2(\Om_\e)},
\end{align*}
if $\e$ is small enough.
\end{proof}

\begin{proof}[Proof of Theorem~\ref{th1.4}]
By the standard bracketing arguments (see, for instance,
\cite[Ch.X\!I\!I\!I, Sec. 15, Prop. 4]{RS4}) we see that the
eigenvalues of $\mathcal{H}_\e^{(p)}-\frac{\tau^2}{\e^2}$ are
estimated from above by those of the same operator in the case
$\eta=\pi/2$. In other words, we increase the eigenvalues of
$\mathcal{H}_\e^{(p)}-\frac{\tau^2}{\e^2}$, if we replace the
Neumann condition on $\mathring{\G}_\e$ by the Dirichlet one. In the
latter case given any $N$ there exists $\e_0>0$ so that for
$\e<\e_0$ the first $N$ eigenvalues are $n^2$ with the
eigenfunctions $\sin n x_2$. Hence,
\begin{equation}\label{3.23}
0\leqslant\l_n(\tau,\e)-\frac{\tau^2}{\e^2}\leqslant n^2,\quad
n\leqslant N,\quad \e<\e_0.
\end{equation}
By \cite[Ch. I\!I\!I, Sec. 1, Th. 1.4]{OIS} and by
Theorem~\ref{th1.3} we have
\begin{equation*}
\left|\frac{1}{\l_n(\tau,\e)-\frac{\tau^2}{\e^2}}-\frac{1}{n^2}\right|
\leqslant \frac{\e+5\e^{1/2}|\ln\sin\eta|^{1/2}}{\d^{1/2}}.
\end{equation*}
The statement of the theorem follows from two last estimates.
\end{proof}

\section{Bottom of the spectrum}

In this section we prove Theorem~\ref{th1.5}. First we prove that
the eigenvalue $\l_1(\tau,\e)$ attains its minimum at $\tau=0$.

In the same way as in the proof of Theorem~\ref{th1.4}, by the
bracketing arguments we see that the eigenvalues of
$\mathcal{H}_\e^p(\tau)$ are estimated from below by those of the
same operator with $\eta=0$, i.e., when we replace the Neumann
condition on $\mathring{\G}_\e$ by the Dirichlet one. The lowest
eigenvalue of the latter operator is
$\frac{1}{4}+\frac{\tau^2}{\e^2}$ and therefore
\begin{align}
&\l_1(\tau,\e)\geqslant \frac{1}{4}+\frac{\tau^2}{\e^2}, \quad
\tau\in[-1,1),\label{4.19}
\\
&\l_1(\tau,\e)\geqslant
\frac{5}{4},\hphantom{,+\frac{\tau^2}{\e^2}}\quad |\tau|\geqslant
\e.\label{4.20}
\end{align}
Since by (\ref{1.8}) the eigenvalue $\l_1(0,\e)$ behaves as
\begin{equation}\label{4.21}
\l_1(0,\e)=1+o(1),\quad \e\to+0,
\end{equation}
in view of (\ref{4.20}) we conclude that $\l_1(\tau,\e)\geqslant
\l_1(0,\e)$ as $|\tau|\geqslant \e$ for sufficiently small $\e$, and
thus
\begin{equation}\label{4.22}
\inf\limits_{\tau\in[-1,1)}\l_1(0,\e)=
\inf\limits_{\tau\in[-\e,\e]}\l_1(0,\e).
\end{equation}

Consider the case $|\tau|\leqslant\e$. For such $\tau$, the
eigenvalue $\l_1(\tau,\e)$ is simple as it follows from (\ref{1.8}).
Let $\psi_\e=\psi_\e(x)$ be the real-valued eigenfunction associated
with $\l_1(0,\e)$ normalized in $L_2(\Om_\e)$.

\begin{lemma}\label{lm4.1}
The convergence
\begin{equation}\label{4.1}
\left\|\frac{\p\psi_\e}{\p x_1}\right\|_{L_2(\Om_\e)}\to0,\quad
\e\to+0,
\end{equation}
holds true.
\end{lemma}
\begin{proof}
By the definition, the function $\psi_\e$ satisfies the identity
\begin{equation}\label{4.2}
\|\nabla\psi_\e\|_{L_2(\Om_\e)}^2=\l_1(0,\e).
\end{equation}
Let $\psi_\e^\bot$ be the projection of $\psi_\e$ on
$\mathfrak{L}^\bot$, and
$\mathring{\psi}_\e:=\psi_\e-\psi_\e^\bot\in \mathfrak{L}$,
$\mathring{\psi}_\e=\mathring{\psi}_\e(x_2)$. By the inequality
(\ref{3.4}) with $\tau=0$ we obtain
\begin{equation*}
\Big\|\frac{\p\psi_\e}{\p x_1}\Big\|_{L_2(\Om_\e)}^2\geqslant
4\e^{-2}\|\psi_\e^\bot\|_{L_2(\Om_\e)}^2.
\end{equation*}
Together with (\ref{4.21}), (\ref{4.2}) it yields
\begin{equation}\label{4.3}
\|\psi_\e^\bot\|_{L_2(\Om_\e)}=\Odr(\e),\quad \e\to+0.
\end{equation}
Since
\begin{equation*}
\|\mathring{\psi}_\e\|_{L_2(\Om_\e)}^2+\|\psi^\bot\|_{L_2(\Om_\e)}^2=
\|\psi_\e\|_{L_2(\Om_\e)}^2=1,
\end{equation*}
it follows that
\begin{equation}\label{4.4}
\|\mathring{\psi}_\e\|_{L_2(\Om_\e)}=1+\Odr(\e),\quad \e\to+0.
\end{equation}

We integrate the equation
\begin{equation*}
 -\D\psi_\e=\l_1(0,\e)\psi_\e
\end{equation*}
w.r.t. $x_1\in(-\e\pi/2,\e\pi/2)$ for $x_2\in(0,\pi)$,
\begin{equation*}
 -\frac{d^2\mathring{\psi}_\e}{dx_2^2}=\l_1(0,\e)\mathring{\psi}_\e,\quad
x_2\in(0,\pi),\qquad\mathring{\psi}_\e(\pi)=0.
\end{equation*}
Hence,
\begin{align*}
&\mathring{\psi}_\e(x_2)=C_\e(\e\pi)^{-1/2}\sin\sqrt{\l_1(0,\e)}(\pi-x_2),
\\
&\|\mathring{\psi}_\e\|_{L_2(\Om_\e)}^2=\frac{C_\e^2}{2} \left( \pi-
\frac{\sin2\sqrt{\l_1(0,\e)}\pi}{2\sqrt{\l_1(0,\e)}}\right),
\end{align*}
where $C_\e$ is a constant. If follows from (\ref{4.4}),
(\ref{4.21}) that
\begin{equation}\label{4.5}
C_\e^2=\frac{2}{\pi}+o(1),\quad \e\to+0.
\end{equation}
By direct calculations we check that
\begin{equation*}
\Big\|\frac{d\mathring{\psi}_\e}{d x_2}\Big\|_{L_2(\Om_\e)}^2=
\frac{C_\e^2\l_1(0,\e)}{2}\left(\pi+\frac{\sin
2\sqrt{\l_1(0,\e)}\pi}{2\sqrt{\l_1(0,\e)}}\right)=\l_1(0,\e)(1+o(1)),\quad
\e\to+0.
\end{equation*}
We substitute this identity and the asymptotics (\ref{4.21}) into
(\ref{4.2}),
\begin{align*}
&\|\nabla\psi_\e^\bot\|_{L_2(\Om_\e)}^2=o(1),\quad 
\Big\|\frac{\p\psi_\e}{\p x_1}\Big\|_{L_2(\Om_\e)}^2
=\Big\|\frac{\p\psi_\e^\bot}{\p
x_1}\Big\|_{L_2(\Om_\e)}^2=o(1),\quad \e\to+0.
\end{align*}
\end{proof}

Applying the bracketing arguments in the same way as above, we can
estimate the eigenvalues $\l_1(\tau,\e)$ and $\l_2(\tau,\e)$ as
\begin{equation}\label{4.6}
\frac{1}{4}\leqslant \l_1(\tau,\e)-\frac{\tau^2}{\e^2}\leqslant 1,
\quad \frac{9}{4}\leqslant
\l_2(\tau,\e)-\frac{\tau^2}{\e^2}\leqslant 4
\end{equation}
for $\e$ small enough, $|\tau|\leqslant \e$. One can make sure
easily that
\begin{align*}
&\left(\mathcal{H}_\e^{(p)}(\tau)-\frac{\tau^2}{\e^2}\right)\psi_\e=
\l_1(0,\e)\psi_\e-\frac{2\iu \tau}{\e}\frac{\p\psi_\e}{\p x_1},
\\
&\left(\left(\mathcal{H}_\e^{(p)}(\tau)-
\frac{\tau^2}{\e^2}\right)\psi_\e,\psi_\e\right)_{L_2(\Om_\e)}=\l_1(0,\e),
\\
&\left\|\left(\mathcal{H}_\e^{(p)}(\tau)-\frac{\tau^2}{\e^2}\right)
\psi_\e\right\|_{L_2(\Om_\e)}^2=\l_1^2(0,\e)+\frac{4\tau^2}{\e^2}
\Big\|\frac{\p\psi_\e}{\p x_1}\Big\|_{L_2(\Om_\e)}^2.
\end{align*}
Employing these formulas, we apply Temple inequality (see \cite[Ch.
4, Sec. 4.6, Th. 4.6.3]{Dav}) to the operator
$\mathcal{H}_\e^{(p)}(\tau)-\frac{\tau^2}{\e^2}$,
\begin{align*}
\l_1(\tau,\e)-\frac{\tau^2}{\e^2}&\geqslant \frac{\frac{9}{4}
\left(\left(\mathcal{H}_\e^{(p)}(\tau)-
\frac{\tau^2}{\e^2}\right)\psi_\e,\psi_\e\right)_{L_2(\Om_\e)} -
\left\|\left(\mathcal{H}_\e^{(p)}(\tau)-\frac{\tau^2}{\e^2}\right)
\psi_\e\right\|_{L_2(\Om_\e)}^2}
{\frac{9}{4}-\left(\left(\mathcal{H}_\e^{(p)}(\tau)-
\frac{\tau^2}{\e^2}\right)\psi_\e,\psi_\e\right)_{L_2(\Om_\e)} }
\\
&=\frac{\frac{9}{4}\l_1(0,\e)-\l_1(0,\e)^2-\frac{4\tau^2}{\e^2}
\big\|\frac{\p\psi_\e}{\p x_1}\big\|_{L_2(\Om_\e)}^2}
{\frac{9}{4}-\l_1(0,\e)}
\\
&=\l_1(0,\e)-\frac{4\tau^2}{\e^2
\left(\frac{9}{4}-\l_1(0,\e)\right)} \Big\|\frac{\p\psi_\e}{\p
x_1}\Big\|_{L_2(\Om_\e)}^2.
\end{align*}
Hence, by Lemma~\ref{lm4.1} and the asymptotics (\ref{4.21})
\begin{equation*}
\l_1(\tau,\e)\geqslant \l_1(0,\e)+\frac{\tau^2}{\e^2}
\left(1-\frac{16}{9-4\l_1(0,\e)}\Big\|\frac{\p\psi_\e}{\p
x_1}\Big\|_{L_2(\Om_\e)}^2\right)\geqslant \l_1(0,\e),\quad
\tau\in[-\e,\e],
\end{equation*}
if $\e$ is small enough.

We proceed to the asymptotics for $\l_1(0,\e)$. We construct it
first formally and then we justify it. The formal constructing is
based on the boundary layer method and in fact it follows the main
ideas of \cite{AA-GRR}.

We construct the asymptotics for $\l_1(0,\e)$ as the series
(\ref{1.9}). The asymptotics for the associated eigenfunction is
constructed as
\begin{align}
&\widetilde{\psi}_\e(x)=\Psi_\e^{\mathrm{in}}(x,\eta)+\chi(x_2)
\Psi_\e^{\mathrm{bl}}(\xi,\eta),\label{4.7}
\\
&\Psi_\e^{\mathrm{in}}(x,\eta)=\sin\sqrt{\l_1(0,\e)}(\pi-x_2),\nonumber
\end{align}
where the cut-off function $\chi$ was defined before
Lemma~\ref{lm2.1} and the variables $\xi$ were introduced in
(\ref{2.7}). In contrast to $\psi_\e$, the function
$\widetilde{\psi}_\e$ is not supposed to be normalized in
$L_2(\Om_\e)$.

The function $\Psi_\e^{\mathrm{bl}}$ is a boundary layer at
$\mathring{\G}_-$ and its asymptotics is sought as
\begin{equation}\label{4.8}
\Psi_\e^{\mathrm{bl}}(\xi,\eta)=\sum\limits_{i=1}^{\infty}\e^i
v_i(\xi,\eta).
\end{equation}
The main aim of the formal constructing is to determine the numbers
$\mu_i$ and the functions $v_i$.

It is clear that
\begin{equation}\label{4.9}
\begin{aligned}
\Psi_\e^{\mathrm{in}}(0,\eta)=&\sin\sqrt{\l_1(0,\e)}\pi=\sum\limits_{i=1}^{\infty}
\e^i
\left(-\frac{\pi}{2}\mu_i+G_i^{(\mathrm{D})}(\mu_1,\ldots,\mu_{i-1})\right),
\\
\frac{d\Psi_\e^{\mathrm{in}}}{dx_2}(0,\eta)=&-\sqrt{\l_1(0,\e)}
\cos\sqrt{\l_1(0,\e)}\pi
\\
=&1+\sum\limits_{i=1}^{\infty}\e^i
\left(\frac{\mu_i}{2}+G_i^{(\mathrm{N})}(\mu_1,\ldots,\mu_{i-1})\right),
\end{aligned}
\end{equation}
where $G_i^{(\mathrm{D})}$, $G_i^{(\mathrm{N})}$ are some
polynomials, and, in particular,
\begin{equation}\label{4.10}
G_1^{(\mathrm{D})}=0,\quad G_1^{(\mathrm{N})}=0,\quad
G_2^{(\mathrm{D})}(\mu_1)=\frac{\pi}{8}\mu_1^2.
\end{equation}
The function $\widetilde{\psi}_\e$ satisfies the boundary condition
on $\mathring{\g}_\e$ and $\mathring{\G}_\e$, and by (\ref{4.8}),
(\ref{4.9}) it implies the boundary conditions for $v_i$,
\begin{align}
&\frac{\p v_1}{\p\xi_2}=-1,\quad \frac{\p
v_i}{\p\xi_2}=-\frac{1}{2}\mu_{i-1}-G_{i-1}^{(\mathrm{N})}
(\mu_1,\ldots,\mu_{i-2}),\quad \xi\in\mathring{\G}(\eta),\quad
i\geqslant 2,\label{4.11}
\\
&v_i=\frac{\pi}{2}\mu_i-G_i^{(\mathrm{D})}(\mu_1,\ldots,\mu_{i-1}),
\quad \xi\in\mathring{\g}(\eta),\quad i\geqslant 1,\label{4.12}
\\
&\mathring{\g}_\eta:=\p\Pi\cap\g(\eta),\quad \mathring{\G}(\eta):=
\p\Pi\cap\G(\eta),\nonumber
\end{align}
where, we remind, the sets $\g(\eta)$ and $\G(\eta)$ were introduced
in (\ref{2.8a}). The functions $v_i$ should satisfy the periodic
boundary conditions on the lateral boundaries of $\Pi$, since the
same is assumed for $\psi_\e$. And they should decay exponentially
as $\xi_2\to+\infty$, since they are boundary layer functions.

In order to obtain the equations for $v_i$, we substitute the series
(\ref{1.9}), (\ref{4.8}) into the equation
\begin{equation*}
-\D\Psi_\e^{\mathrm{bl}}=\l_1(0,\e)\Psi_\e^{\mathrm{bl}},\quad
x\in\Om_\e,
\end{equation*}
pass to the variables $\xi$, and equate the coefficients of the same
powers of $\e$. It implies
\begin{equation}\label{4.13}
-\D_{\xi}v_i=\sum\limits_{j=0}^{i-3} \mu_j v_{i-j-2},\quad
\xi\in\Pi,
\end{equation}
where $\mu_0:=1$. The functions $v_1$, $v_2$ are harmonic ones and
we can find them explicitly,
\begin{equation*}
v_1=X,\quad v_2=\frac{\mu_1}{2}X,
\end{equation*}
where, we remind, the function $X$ was introduced in (\ref{2.7}). It
follows from (\ref{2.8}) that
\begin{equation*}
v_1=\ln\sin\eta,\quad v_2=\frac{\mu_1}{2}\ln\sin\eta,\quad
\xi\in\mathring{\g}(\eta),
\end{equation*}
and by (\ref{4.10}), (\ref{4.11}), (\ref{4.12}) we obtain
\begin{equation*}
\frac{\pi}{2}\mu_1=\ln\sin\eta,\quad \frac{\mu_1}{2}\ln\sin\eta=
\frac{\pi}{2}\mu_2-\frac{\pi}{8}\mu_1^2.
\end{equation*}
The identity obtained lead us directly to the formulas (\ref{1.10}).

The solvability condition of the problems (\ref{4.13}),
(\ref{4.11}), (\ref{4.12}) for $i\geqslant 3$ is given by Lemma~3.1
in \cite{AA-GRR},
\begin{align}
& \pi \left(\frac{\pi}{2}\mu_i-G_i^{(\mathrm{D})}-
\left(\frac{1}{2}\mu_{i-1}+G_{i-1}^{(\mathrm{N})}\right) \ln\sin\eta
\right) =\sum\limits_{j=0}^{i-3} \mu_j\int\limits_\Pi Y
v_{i-j-2}\di\xi, \label{4.14}
\\
&Y=Y(\xi,\eta):=X(\xi,\eta)+\xi_2-\ln\sin\eta.\nonumber
\end{align}
It implies the formulas for $\mu_i$,
\begin{equation}\label{4.15}
\begin{aligned}
\mu_i=\frac{2}{\pi} \Bigg( \frac{1}{\pi} &\sum\limits_{j=0}^{i-3}
\mu_j\int\limits_\Pi Y v_{i-j-2}\di \xi+
G_i^{(\mathrm{D})}(\mu_1,\ldots,\mu_{i-1})
\\
&+\left( \frac{\mu_{i-1}}{2} +
G_{i-1}^{(\mathrm{N})}(\mu_1,\ldots,\mu_{i-2})\right)\ln\sin\eta
\Bigg).
\end{aligned}
\end{equation}
So, the problems (\ref{4.13}), (\ref{4.11}), (\ref{4.12}) are
solvable.

By induction it follows from \cite[Lemma~3.7]{VMU-B} that
\begin{equation*}
\int\limits_{-\frac{\pi}{2}}^{\frac{\pi}{2}}
v_i(\xi,\eta)\di\xi_1=0\quad\text{for all}\quad \xi_2>0.
\end{equation*}
It allows us to apply Theorem~3.1 from \cite{VMU-B} to the problems
(\ref{4.13}), (\ref{4.11}), (\ref{4.12}). Similar fact was the core
of the proof of Lemma~4.2  in \cite{VMU-B}, and repeating
word-by-word the proof of this lemma, we arrive at
\begin{lemma}\label{lm4.2}
As $\eta\to+0$, the identities (\ref{1.11}) and the uniform in
$\eta$ estimates
\begin{align*}
&\|\xi_2^p v_j\|_{L_2(\Pi)}\leqslant C|\ln\eta|^{j-1},\quad
\Big\|\xi_2^{p+1}\nabla_\xi \frac{\p
v_j}{\p\xi_2}\Big\|_{L_2(\Pi)}\leqslant C|\ln\eta|^{j-1}
\\
&\Big\|\xi_2^m\nabla_\xi \frac{\p^m v_j}{\p
\xi_2^m}\Big\|_{L_2(\Pi)}\leqslant C|\ln\eta|^{j-\frac{1}{2}},\quad
m=0,1,
\end{align*}
hold true, where $p\geqslant 0$.
\end{lemma}

Employing this lemma and reproducing the proof of Lemma~5.1 in
\cite{VMU-B}, one can prove easily

\begin{lemma}\label{lm4.3}
For any $p\geqslant 2$, $R>1$ the uniform in $R$ and $\eta$
estimates
\begin{equation*}
\|v_j\|_{L_2(\Pi_R)}\leqslant CR^{-p+1} (|\ln\eta|^{j-1}+1),\quad
\|\nabla_\xi v_j\|_{L_2(\Pi_R)}\leqslant CR^{-p+1}
(|\ln\eta|^{j-1}+1)
\end{equation*}
hold true.
\end{lemma}

Given $M\geqslant 2$, denote
\begin{equation*}
\L^{(M)}_\e:=1+\sum\limits_{i=1}^{M} \e^i \mu_i(\eta), \quad
\widetilde{\Psi}_\e^{(M)}(x):=
\sin\sqrt{\L^{(M)}_\e}(\pi-x_2)+\chi(x_2)\sum\limits_{i=1}^{M} \e^i
v_i(\xi,\eta).
\end{equation*}

\begin{lemma}\label{lm4.4}
The function $\widetilde{\Psi}_\e^{(M)}\in\H^1(\Om_\e)$ satisfies
the periodic boundary condition on the lateral boundaries of
$\Om_\e$ and is a generalized solution to the problem
\begin{align*}
-&\D\widetilde{\Psi}_\e^{(M)}=\L_\e^{(M)}\widetilde{\Psi}_\e^{(M)}
+\widetilde{f}_{\e}^{(M)}\quad \text{in}\quad \Om_\e,\quad
\widetilde{\Psi}_\e^{(M)}=0\quad \text{on}\quad \mathring{\G}_+,
\\
&\widetilde{\Psi}_\e^{(M)}=B^{(M)}_{\e,\mathrm{D}}\quad
\text{on}\quad \mathring{\g}_\e,\qquad
\frac{\p\widetilde{\Psi}_\e^{(M)}}{\p
x_2}=B^{(M)}_{\e,\mathrm{N}}\quad \text{on}\quad \mathring{\G}_\e,
\end{align*}
where $\widetilde{f}_\e^{(M)}\in L_2(\Om_\e)$,
$B^{(M)}_{\e,\mathrm{D}}$, $B^{(M)}_{\e,\mathrm{N}}$ are constants.
The uniform in $\e$ and $\eta$ estimates
\begin{equation}\label{4.16}
\begin{aligned}
&\|\widetilde{f}_\e^{(M)}\|_{L_2(\Om_\e)}\leqslant
C\e^{M}(|\ln\eta|^{M-2}+1),
\\
&|B^{(M)}_{\e,\mathrm{N}}|\leqslant C\e^{M}(|\ln\eta|^M+1),\quad
|B^{(M)}_{\e,\mathrm{D}}|\leqslant C\e^{M+1}(|\ln\eta|^{M+1}+1),
\end{aligned}
\end{equation}
hold true.
\end{lemma}
This lemma can be checked by direct calculations with employing
Lemma~\ref{lm4.3}.

We let
\begin{equation*}
\Psi_\e^{(M)}(x):=\widetilde{\Psi}_\e^{(M)}(x)-\chi(x_2)
\big(B^{(M)}_{\e,\mathrm{D}}+x_2 B^{(M)}_{\e,\mathrm{N}}\big).
\end{equation*}
In view of Lemma~\ref{lm4.4}, this function belongs to the domain of
$\mathcal{H}_\e^{(p)}(0)$ and
\begin{equation}\label{4.17}
\left(\mathcal{H}_\e^{(p)}(0)-\L_\e^{(M)}\right)\Psi_\e^{(M)}=f_\e^{(M)},\quad
\|f_\e^{(M)}\|_{L_2(\Om_\e)}\leqslant
C\e^{M-\frac{3}{2}}|(\ln\eta|^{M-2}+1).
\end{equation}
It also easy to check that
\begin{equation}\label{4.18}
\|\Psi_\e^{(M)}\|_{L_2(\Om_\e)}=\e^{1/2}\left(
\frac{\pi}{\sqrt{2}}+\Odr\big(\e(|\ln\eta|+1)\big)\right).
\end{equation}
Denote
\begin{equation*}
\widehat{\Psi}_\e^{(M)}:=\frac{\Psi_\e^{(M)}}{\|\Psi_\e^{(M)}\|_{L_2(\Om_\e)}},
\end{equation*}
and by (\ref{4.17}) and (\ref{3.0a}) we have
\begin{equation*}
\frac{1}{\L_\e^{(M)}}\widehat{\Psi}_\e^{(M)}-
\big(\mathcal{H}_\e^{(p)}(0)\big)^{-1} \widehat{\Psi}_\e^{(M)} =
\widehat{f}_\e^{(M)},\quad
\|\widehat{f}_\e^{(M)}\|_{L_2(\Om_\e)}\leqslant
C\e^{M-2}|(\ln\eta|^{M-2}+1).
\end{equation*}
Since the operator $\big(\mathcal{H}_\e^{(p)}(0)\big)^{-1}$ is
self-adjoint and compact, by \cite[Ch. I\!I\!I, Sec. 1, Lm.
1.1]{OIS} we conclude that there exists an eigenvalue $\l_\e$ of the
operator $\mathcal{H}_\e^{(p)}(0)$ such that
\begin{equation*}
\left|\frac{1}{\L_\e^{(M)}} - \frac{1}{\l_\e}\right|\leqslant
C\e^{M-2}(|\ln\eta|^{M-2}+1).
\end{equation*}
In view of the asymptotics (\ref{4.21}), the only eigenvalue of
$\mathcal{H}_\e^{(p)}(0)$ which satisfies this inequality is
$\l_1(0,\e)$ and $\l_\e=\l_1(0,\e)$. Hence,
\begin{equation*}
|\L_\e^{(M)}-\l_\e|\leqslant C\e^{M-2}(|\ln\eta|^{M-2}+1),
\end{equation*}
and the asymptotics (\ref{1.9}) is proven.

\section*{Acknowledgments}

The authors are grateful to S.E. Pastukhova for discussion of the
results and useful references. A part of this work was done during
the visit of D.B. to the Department of Civil Engineering of Second
University of Naples and it was partially supported by project
"Asymptotic analysis of composite materials and thin and
non-homogeneous structures" (Regione Campania, law n.5/2005). He
is grateful for the warm hospitality extended to him.

D.B. is partially supported by RFBR (No. 09-01-00530) and by the
grants of the President of Russian Federation for young scientists
(MK-964.2008.1) and for leading scientific schools
(NSh-2215.2008.1). He also gratefully acknowledges the support
from Deligne 2004 Balzan prize in mathematics.


\begin{thebibliography}{99}


\bibitem[B]{Bir1}
M.Sh. Birman 2004 On homogenization procedure for periodic operators
near the edge of an internal gap  \textit{St. Petersburg Math. J.}
15  507-513

\bibitem[BS1]{BS4}
M.Sh. Birman, T.A. Suslina 2006 Homogenization of a multidimensional
periodic elliptic operator in a neighbourhood of the edge of the
internal gap \textit{J. Math. Sciences} 136 3682-3690.


\bibitem[BS2]{BS2}
M.Sh. Birman, T.A. Suslina 2006 Homogenization with corrector term
for periodic elliptic differential operators \textit{St. Petersburg
Math. J.} 17 897-973


\bibitem[BS3]{BS5}
M.Sh. Birman, T.A. Suslina 2007 Homogenization with corrector for
periodic differential operators. Approximation of solutions in the
Sobolev class $H^1({\mathbb{R}}^d)$ \textit{St. Petersburg Math. J.}
18 857-955



\bibitem[Bo1]{JMS-B}
D.I. Borisov 2006 On the spectrum of a Schroedinger operator
perturbed by a rapidly oscillating potential \textit{J. Math.
Sciences} 139 6243-6322

\bibitem[Bo2]{BoAA08}
D. Borisov 2009 Asymptotics for the solutions of elliptic systems
with rapidly oscillating coefficients \textit{St. Petersburg Math.
J.} 20 175-191


\bibitem[Bo3]{JPA-B}
D. Borisov 2007  On the spectrum of two quantum layers coupled by a
window  \textit{J. Phys. A} 40 5045-5066

\bibitem[Bo4]{MZ-B} D. Borisov 2001
Two-parameter asymptotics in a boundary-value problem for the
Laplacian \textit{Math. Notes} 70 471-485


\bibitem[Bo5]{VMU-B} D. Borisov 2002
Two-parametrical asymptotics for the eigenevalues of the Laplacian
with frequent alternation of boundary conditions \textit{Vestnik
Molodyh Uchenyh. Seriya Prikladnaya matematika i mehanika (Journal
of young scientists. Applied mathematics and mechanics)} 1 36-52 (in
Russian)


\bibitem[Bo6]{B-MS06}
D. Borisov 2006 Discrete spectrum of a pair of non-symmetric
waveguides coupled by a window \textit{Sb. Math.} 197 475-504

\bibitem[BEG]{BEG-02}
D. Borisov, P. Exner and R. Gadyl'shin 2002 Geometric coupling
thresholds  in a two-dimensional strip \textit{J. Math. Phys.} 43
6265-6278


\bibitem[Bo7]{Izv03-B}
D.I. Borisov 2003 Asymptotics and estimates for the eigenelements of
the Laplacian with frequently alternating nonperiodic boundary
conditions \textit{Izv. Math.} 67 1101-1148

\bibitem[BLP]{BLP}
A. Brillard, M. Lobo, E. P\`erez 1990 Homog\'en\'eisation de
fronti\`eres par \'epi-convergence en \'elasticit\'e lin\'eaire
\textit{Mod\'elisation math\'ematique et Analyse num\'erique} 24
5-26

\bibitem[BGRS]{BGRS}
W. Bulla, F. Gesztesy, W. Renger, B. Simon 1997 Weakly coupled bound
states in quantum waveguides \textit{Proc. Amer. Math. Soc.} 12
1487-1495



\bibitem[C]{Ch}
G.~A.~Chechkin 1994 Averaging of boundary value problems with
singular perturbation of the boundary conditions \textit{Russian
Acad. Sci. Sb. Math.} 79 191-220


\bibitem[DT]{Dlcr}  A.~Damlamian, and Li Ta-Tsien (Li Daqian) 1987
Boundary homogenization for ellpitic problems \textit{J. Math. Pure
et Appl.} 66 351-361



\bibitem[D]{Dav} E.B. Davies 1995 \textit{Spectral theory and differential
operators} (Cambridge Univ. Press)

\bibitem[DK]{DK}
J. Dittrich and J. K\v{r}\'\i\v{z} 2002 Bound states in straight
quantum waveguide with combined boundary condition \textit{J. Math.
Phys.} 43 3892-3915


\bibitem[ESTV]{ESTV}
P. Exner, P. \v Seba, M. Tater, D. Van\v ek 1996 Bound states and
scattering in quantum waveguides coupled laterally through a
boundary window \textit{J. Math. Phys.} 37 4867-4887

\bibitem[EV]{EV2} Exner P. and Vugalter S 1997 Bound-state asymptotic
estimate for window-coupled Dirichlet strips and layers \textit{J.
Phys. A.} 30 7863-7878



\bibitem[FHY]{Fr} A.~Friedman, Ch.~Huang and J.~Yong 1995 Effective
permeability of the boundary of a domain \textit{Commun. Part. Diff.
Eqs}  20 59-102

\bibitem[FS]{FS}
L. Friedlander and M. Solomyak 2007 On the spectrum of the Dirichlet
Laplacian in a narrow strip  \textit{Israel J. Math.}, to appear.
Preprint: \texttt{arXiv:0705.4058}


\bibitem[G1]{AA-GRR} R.R. Gadyl'shin 1999
On the eigenvalue asymptotics for periodically clamped membranes
\textit{St. Petersbg. Math. J.} 10 1-14


\bibitem[G2]{G} Gadyl'shin R 2004 On regular and singular perturbation
of acoustic and quantum waveguides \textit{C.R. Mechanique} 332
647-652

\bibitem[HTWK]{HTW} Hirayama Y., Tokura Y., Wieck A.D., Koch S.,
Haug R.J., K.~von~Klitzing, Ploog K 1993 Transport characteristics
of a window-coupled in-plane-gated wire system \textit{Phys. Rev.
B.} 48 7991-7998

\bibitem[K]{Ku} P. Kuchment 1993 \textit{Floquet theory for partial differential
equations} (Operator Theory: Advances and Applications. Basel:
Birkha\"user Verlag)

\bibitem[LU]{Ld}
O.~A. Ladyzhenskaya and N.~N. Uraltseva 1968 \emph{Linear and
quasilinear elliptic equations} (New York: Academic Press)

\bibitem[OSI]{OIS}  O.~A. Olejnik, A. S. Shamaev and G. A. Yosifyan 1992
\emph{Mathematical problems in elasticity and homogenization.}
(Studies in Mathematics and its Applications. 26. Amsterdam etc.:
North-Holland)


\bibitem[PT]{PT} S.E. Pastukhova and R.N. Tikhomirov 2007
Operator Estimates in Reiterated and Locally Periodic Homogenization
\textit{Dokl. Math.} 76 548-553

\bibitem[P]{Pas}
S.E. Pastukhova 2006 Some Estimates from Homogenized Elasticity
Problems \textit{Dokl. Math.} 73 102-106

\bibitem[RS1]{RS} M. Reed and B. Simon 1980 \textit{Methods Of
Modern Mathematical Physics I: Functional Analysis} (New York:
Academic Press)


\bibitem[RS2]{RS4} M. Reed and B. Simon 1978 \textit{Methods Of
Modern Mathematical Physics IV: Analysis of operators} (New York:
Academic Press)


\bibitem[Z1]{Zh3} V.V. Zhikov 2005
 On operator estimates in homogenization theory
\textit{Dokl. Math.} 72 534-538


\bibitem[Z2]{Zh4} V.V. Zhikov 2006
Some estimates from homogenization theory \textit{Dokl. Math.} 73
96-99


\bibitem[ZPT]{Zh5}
V.V. Zhikov, S.E. Pastukhova, and S. V. Tikhomirova 2006 On the
homogenization of degenerate elliptic equations \textit{Dokl. Math.}
74 716-720


\end{thebibliography}
\end{document}